\documentclass{amsart}
\usepackage{graphicx}
\usepackage{indentfirst,csquotes}

\usepackage{amssymb,amsthm,amsmath}
\usepackage{xcolor,paralist,hyperref,fancyhdr,etoolbox}
\usepackage{algorithm}
\usepackage{algpseudocode}
\usepackage{float}
\usepackage{tabularx} 
\usepackage{multirow}
\newtheorem{theorem}{Theorem}[section]
\newtheorem{definition}[theorem]{Definition}

\hypersetup{ colorlinks=true, linkcolor=black, filecolor=black, urlcolor=black }

\begin{document}
\title{A New Level Set Formulation for Improved Dirichlet Eigenvalue Minimizers} 
\author{Atharv Thakur}
\address{Address}
\email{atharv.thakur@rutgers.edu}

\begin{abstract}
This paper makes several improvements to existing level set based approaches to computing shape optimizers for the Dirichlet eigenvalues subject to a volume constraint. The most notable changes in formulation include an overhaul of the classical level set construction and root-finding procedures as well the use of a regularized approximation to the standard objective function. Our resulting computational minimizers are either comparable to or improvements on the best known minimizers from the literature. We conclude with a survey of subproblems within the field that may benefit from numerical experiments; these include the existence of cusps on the boundary, the end-behavior of eigenfunction weights in the p-parameterized problem, and the nature of Weyl asymptotics as they relate to the Pólya conjecture.
\end{abstract}
\maketitle

\bigskip

\pagenumbering{arabic}

\section{Introduction}
\bigskip

Let $\Omega \subset \mathbb{R}^N$ be an open bounded set. The spectrum of the Laplace operator with Dirichlet boundary conditions on $\Omega$ is of the form $0 < \lambda_{1}(\Omega) \leq \lambda_{2}(\Omega) \leq \space  ...$ with corresponding eigenfunctions $u_{1},u_{2},...$ chosen to form an orthonormal basis for $L^2(\Omega)$. These solve the problem  
\\
\begin{equation}
\begin{cases}
-\Delta u_{k} = \lambda_{k}(\Omega)u_{k} & \text{in } \Omega,
\\
u_k = 0 & \text{on } \partial \Omega.
\end{cases}
\end{equation}
\\
A classical shape optimization problem asks which domain $\Omega$ minimizes $\lambda_k(\Omega)$ subject to a volume constraint. For $\lambda_1$, Faber \cite{Faber1923} and Krahn \cite{Krahn1925} independently proved the Rayleigh Conjecture: the minimizer under a volume constraint is a ball. For $\lambda_2$, the minimizer was proven to be a disjoint union of two identical balls \cite{Krahn1926}. For $k \geq 3$, however, the question remains open. Numerical studies have attempted to remedy this gap, among which the work of Oudet \cite{Oudet2010} often serves as foundation and comparison point for later works. Oudet employed the widely-used finite element method (FEM) for solving the eigenproblem, but notably also utilized level set techniques \cite{OsherSethian1988} for the task of domain representation. Later authors who built on that work, such as Osting and Kao, further refined the application of these techniques and used them to study the minimization of functionals dependent on the Dirichlet eigenvalues of the Laplacian \cite{OstingKao2013}. Many modern works now focus on developing specialized solvers for the eigenproblem, such as the Method of Fundamental Solutions as used in the work of Antunes and Freitas \cite{AntunesFreitas2012} or the boundary integral method as seen in more recent papers by Osting and Kao \cite{OstingKao2014}. Instead of refining the forward solver, however, this work improves upon the domain representation many mesh-based solvers depend on. 
\\\\
Gradient-based optimization schemes for this problem typically require  parameterizing or approximating points on the boundary. The classical level set method takes the latter approach, requiring rootfinding at each step to approximately align its implicit boundary with an underlying fixed grid. Although using a level set function for domain representation has the advantage of not fixing the topology of the domain, the use of classical level set methods result in a corresponding tradeoff in boundary resolution accuracy. Osting and Kao named the development of a level set method that does not require rootfinding like this a challenging extension of their work \cite{OstingKao2014}. This paper proposes a methodology that addresses that challenge in the context of mesh-based forward solvers.     
\bigskip
\section{Iterative Procedure}
\bigskip
When attempting to minimize the Dirichlet eigenvalues subject to a volume constraint, a common objective function is
\\
\begin{equation}
J(\Omega) \equiv |\Omega|^{\frac{2}{N}}\lambda_k(\Omega).
\end{equation}
\\
Minimizing this function is convenient due to its dilation invariance, which is a simple result of the following properties:
\\
\begin{equation}
\lambda_k(t\Omega) = t^{-2}\lambda_k(\Omega),
\end{equation} 
\begin{equation}
    |t\Omega| = t^N|\Omega|,
\end{equation}
\\
where $t\Omega$ represents the image of $\Omega$ by a homothety of ratio t.
\\\\
However, from a computational perspective, $J$ struggles with the minimization of non-simple eigenvalues. When $\lambda_k = \lambda_{k-1}$, $J$ is no longer differentiable with respect to perturbations of the domain. Since our iterative techniques rely on such a derivative, this lack of regularity is inconvenient. Furthermore, since the optimal second eigenvalue $\lambda_2^*$ is of multiplicity $2$, and the multiplicities of higher eigenvalues are largely unknown, we know this issue cannot simply be ignored. Thus, to maintain differentiability, we instead select the following as our objective function for the purposes of gradient-based optimization:
\\
\begin{equation}
    J_p(\Omega) \equiv |\Omega|\left(\sum_{m=1}^{k} \lambda_m^p\right)^{\frac{1}{p}}.
\end{equation}
\\
The use of a $p$-norm with an undecided parameter introduces some complexities that we will discuss later; for now, it is sufficient to note that $J_p$ acts as a more robust version of $J$ given a sufficiently large $p$ in a computational setting. It is clear that, like $J$, $J_p$ is dilation invariant in two dimensions based on properties (3) and (4). Therefore, the task of minimizing $J_p$ is equivalent up to a rescaling to the problem
$$\min_{|\Omega|=1}\left(\sum_{m=1}^{k}\lambda_m^p\right)^{\frac{1}{p}}.$$
\\
Proving that the volume constrained minimizers of $J_p$ converge to the minimizers of $J$ as $p \to \infty$ takes a little more care. We will prove this later in Theorem 3.5 with a mild domain constraint. 
\\\\
Unfortunately, Theorem 3.5 only gives us a qualitative argument for the convergence of minimizers, as there is broadly a lack of information on the rate of convergence or asymptotic behavior of the $p$-parameterized problem within the field. These present interesting theoretical questions, and we will later provide some relevant numerical experiments in Section 5. However, for the primary goals of this work, we must simply select a $p$ large enough to make errors in approximation negligible. The nuances of choosing exact values of $p$ will be expanded upon later when our computational mechanisms become more apparent. 
\\\\
In line with that primary goal of minimizing $J$, we will only save a candidate minimizer during the iteration process if it minimizes the computed value of $J$, not $J_p$. The benefits of nevertheless selecting $J_p$ as our objective function come from the increased information held in its derivative. To illustrate this point, we must introduce the notion of the shape derivative. Given a vector field $V$ of $\mathbb{R}^N$ known as the deformation field, we can construct derivatives of shape-dependent quantities with respect to deformations in the domain. For us, these are
\\
\begin{equation}
    d(\lambda_k(\Omega))(\Omega,V) = - \int_{\partial\Omega}\left(\frac{\partial u_k}{\partial n}\right)^2\;V \cdot n \; \text{d}\sigma,
\end{equation}
\begin{equation}
    d(|\Omega|)(\Omega,V)=\int_{\partial\Omega}V\cdot n\;\text{d}\sigma,
\end{equation}
\\
where $n$ is the normal vector field of $\partial\Omega$ and $u_k$ is the $k$th eigenfunction normalized such that $||u_k||_{L^2(\Omega)}=1$. Note that, in (6), the eigenvalue $\lambda_k$ must be simple.
\\\\
Applying the shape derivative to $J_p$ gives us an equation of the form
\\
\begin{equation}
    d(J_p(\Omega))(\Omega,V)=c\int_{\partial \Omega} \sum_{m=1}^{k} \left[\mu_m\ \left(\frac{\partial u_m}{\partial n}\right)^2-1\right] V \cdot n\; d\sigma,
\end{equation}
\\
where $c$ and $\mu_m$ are constants dependent on $|\Omega|, \lambda_m,$ and $p$. 
\\\\
Fortunately, (8) remains valid even if the eigenvalues are not simple, giving it a clear advantage over the shape derivative of $J$. This is due to the fact that (\theequation) is invariant under orthonormal changes of basis for eigenfunctions. We will prove this later in Theorem 3.1.
\\\\
Moving forward, we remark that, with any fixed $p$, (\theequation) assigns equal weights $\mu_m$ to the eigenfunctions associated with the uppermost non-simple eigenvalue. The weights assigned to lower eigenfunctions, meanwhile, decay as $p$ increases. To be precise, if $W_p$ are the minimizers of $J_p$, $\mu_{m} \to 0 $ if $\lambda_m(W_p) \to \lambda < \lim_{p\to\infty} \lambda_k(W_p)$, otherwise, $\mu_m\in [0,1]$. Selections for $p$ in a computational environment depend on several factors, including a priori estimates on the error of the forward solver and the  proximity of the current shape to a local minimum. In this work, we chose $p=10^2$ for initial iterations due to experimentally faster convergence rates and the ability to use standard 64 bit floating point arithmetic instead of variable precision arithmetic. Then, during later iterations of refinement, we use $p\in[10^3,3\cdot10^3]$. Although the latter range may seem small, it results in functions $J_p$ that exhibit all of the behaviors we desire from our regularization: lower eigenfunction weights become practically negligible, but upper eigenfunctions with $\lambda_m\approx \lambda_k$ are given similar weights $\mu_M \approx \mu_k$. Since numerical resolution of Dirichlet eigenvalues rarely results in true multiplicity, we must intentionally leave some room for error. \\\\
Computational experiments strongly evidenced the advantages of selecting $J_p$ as our objective function; not only did $J_p$ result in more predictable handling of non-simple eigenvalues, but it also facilitated smoother domain deformations, especially during changes in topology. 
\\\\
The natural next question at hand is the choice of domain representation. Level set methods, as popularized by Osher and Sethian in their 1988 seminal work on front tracking \cite{OsherSethian1988}, are a way to describe the time-dependent motion of an implicit surface as governed by the Hamilton-Jacobi equation. A level set representation of a domain is stored implicitly in the zero set of a so-called level set function (LSF) $\phi$ on that domain. Classically, LSFs are initialized to the signed distance function on a fixed Cartesian grid, and then periodically reinitialized to that state as they lose important properties during the iteration process \cite{OsherFedkiw2003}. An LSF must have stable behavior near the interface and directionally meaningful gradients without sharp variation. Signed distance functions seem to exemplify these properties due to their constant magnitude gradients and relative smoothness. Perhaps more importantly, they are an especially natural choice in the context of their classical use-case: geometric flows. This is due to the fact that their derivatives easily lend to computing surface curvature.     
\\\\
Unfortunately, classical conventions also present two major issues: the resolution of the boundary to a fixed grid and the re-initialization of the signed distance function. These sources of error amplify one another, and extensive studies have been conducted on mitigating the latter \cite{OsherFedkiw2003}. Use cases that are more sensitive to changes in $\phi$, such as our own, require more frequent re-initialization to preserve the necessary characteristics of an LSF across updates. In initial testing with a classical level set approach, it was found that the extreme case of re-initializing every iteration was optimal for boundary tracking in our application.
\\\\
Performing re-initialization this frequently is expensive, and results in compounding errors when resolving the boundary to a fixed grid. It is possible to scale the grid size to mitigate this issue, but doing so significantly increases computational complexity while providing diminishing returns. Furthermore, regardless of size, a fixed Cartesian grid inherently cannot capture the full details of our rather complicated boundary, which is composed of unknown smooth graphs, singular points in higher dimensions, and "cusp" points \cite{kriventsov2017regularitydegenerate} --- the last of which we will discuss in Section 5. 
\\\\
Addressing these concerns brings us to the key new developments of this paper: a grid-less level set approach, and a new choice of LSF we will henceforth call the Laplace LSF. We will start with an introduction to the latter.
\\\\
Let $\phi$ be the solution to the modified torsion problem

\begin{equation}
\begin{cases}
\Delta \phi = 1 & \text{in } \Omega,
\\
\phi = 0 & \text{on } \partial \Omega.
\end{cases}
\end{equation}
\\
Then, $\phi$ naturally defines a desirable LSF for $\Omega$. It is $0$ on the boundary, smooth, and has boundary gradients aligned with the boundary normal vectors. More significantly, this formulation for $\phi$ provides powerful application-specific benefits. Since our forward solver is the mesh-based FEM, the majority of our computational time is spent creating meshes for the current $\Omega$. Once the mesh has been created, however, it can be freely reused to solve multiple PDEs. Thus, computing $\phi$ as per (\theequation) is dramatically subsidized by the pre-existence of the mesh used to solve the eigenproblem. Although meshing is initially expensive, solvers that operate on existing meshes --- such as the implicitly restarted Lanczos method used in this work --- are then able to exhibit strong performance on large, sparse matrices like the FEM matrices \cite{calvetti1994}. Furthermore, in exchange for the initial computational investment, the Delaunay triangulations we use for meshing provide a detailed domain representation that we can take full advantage of using this new formulation for $\phi$.
\\\\
Another benefit of this approach comes from the process by which the LSF quantities, $\phi$ and $\nabla \phi$, and the eigenfunction quantities, $u_1,u_2,...u_k$ and $\nabla u_1,\nabla u_2,...\nabla u_k$, are computed. Using a classical LSF entails computing these functions on differing and likely mismatched sets of nodes, which can lead to significant and compounding interpolative error. In our new approach, however, both types of functions are conveniently calculated on the same set of mesh nodes, resulting in greater accuracy for our shape derivatives and other necessary quantities.
\\\\
Our new definition of $\phi$ is more natural to the geometry of the problem and provides superior accuracy for our use case. Nevertheless, its benefits quickly become restrictions if we must stick to a static domain representation. As a consequence of reusing the FEM mesh, $\phi$ is initially defined only on $\Omega$; there is no ambient grid with points in its complement, $\Omega^c$. Thus, if our $\phi$ was used in the same manner as a classical LSF, the interface would be prevented from expanding. Furthermore, permanently fixing grid points to initial nodal locations, even with some kind of extension mechanism to $\Omega^c$, would significantly hinder the representation of new boundaries with different curvature. This leads us to the introduction of the grid-less level set. 
\\\\
The grid-less level set is a dynamic point cloud defined within a narrow band containing the interface, as is standard numerical practice for classical level sets \cite{OsherFedkiw2003}. This band extends an equal distance into the interior and exterior of $\Omega$ and is composed of multiple layers. For this application, each layer is expanded or contracted in the normal direction from $\partial \Omega$ and separated by $h_{min}$, the smallest edge size in the triangulation mesh. Experimentally, we determined that using a single internal and external layer, for a total of three layers including $\partial \Omega $, provided the best performance and mitigated extrapolation errors. The exact mechanisms of this grid-less approach are best understood through its rootfinding and update procedures. However, to reach that point, we must first the describe the advection process. As usual, the dynamics of the interface are governed by a Hamilton-Jacobi equation of the form:
\\
\begin{equation}
    \frac{\partial \phi}{\partial t} + \vec{V}\cdot\nabla\phi = 0.
\end{equation}
\\
Since $\nabla \phi$ is by construction always normal to $\partial \Omega$, and we are only interested in deformations normal to $\partial \Omega$, we can rewrite this in the more useful form:
\\
\begin{equation}
    \frac{\partial \phi}{\partial t} +V_n|\nabla \phi| = 0.
\end{equation}
\\
Given an arbitrary objective function $J $, we can express its variation with respect to a boundary deformation by $V$ as:
\\
\begin{equation}
    d(J(\Omega)(\Omega, V)) = {\left \langle \frac{\partial J}{\partial V_n}, V_n \right \rangle}_{L^2(\partial \Omega)}.
\end{equation}
\\
Then, to first order, the optimal choice of $V_n$ agrees with $-\frac{\partial J}{\partial V_n}$, i.e.
\\
\begin{equation}
    V_n \big|_{\partial \Omega} = -\frac{\partial J}{\partial V_n}.
\end{equation}
See \cite{OstingKao2013} for proofs of these propositions.
\\\\
Together, these allow us to determine $V_n$ on $\partial \Omega$ for an arbitrary objective function, which in our case we choose to be $J_p$. This gives us the following formula on the boundary:
\\
\begin{equation}
    V_n=-\left(\sum_{i=1}^{k} \lambda_i^p\right)^{\frac{1}{p}}+|\Omega|\left(\sum_{i=1}^{k} \lambda_i^p\right)^{\frac{1-p}{p}}\cdot\left(\sum_{j=1}^{k}\lambda_j^{p-1}\cdot\left(\frac{\partial u_j}{\partial n}\right)^2\right).
\end{equation}
\\
This is merely an expansion of the integrand of (8), i.e.
$$V_n=c\ \sum_{m=0}^{k} \left[\mu_m\ \left(\frac{\partial u_m}{\partial n}\right)^2-1\right]. $$
\\
$V_n$ must now be extended to the entire narrow band where $\phi$ is defined. This is done in a similar manner to \cite{OstingKao2013}. We have two cases:
\begin{enumerate}
    \item The points in question are on the interior of $\Omega$. We denote this region $\Omega^-$. Since the eigenfunctions are defined on this region, we can evaluate the formula as-is. The eigenfunctions are interpolated using the FEM nodal basis.
    \item The points in question are on the exterior of $\Omega$. We denote this region $\Omega^+$. $V$ is extended via a custom $\Theta(n\log n)$ algorithm that assigns each $V_n(\mathbf{x})$ to $V_n(\mathbf{x_o})$, where $\mathbf{x_0}$ is the closest point on $\partial \Omega $ to $\mathbf{x}$ within a mesh-size dependent interval along the boundary. This approach takes advantage of the natural ordering of discrete samples from $\partial \Omega$ within the implementation of the greater program to create a highly efficient procedure. We disregard global comparisons here, as our primary goal is to create locally smooth extensions normal to $\partial \Omega$ while minimizing computational overhead. 
\end{enumerate}

We must also extend $\phi$ and $\nabla \phi$ to the narrow band. This is done  through linear interpolation or extrapolation by using nodal values as our basis. As mentioned earlier, since our points of interest for both the eigenfunctions and the LSF are either shared nodes on $\partial \Omega$, or points where the interpolation is based on the same subset of closest nodes, these approximations will have less error with respect to one another. Furthermore, since layer distances are restricted to the FEM mesh $h_{min}$, it is unlikely first order approximations will generate disproportionate error. 
\\\\
Now comes our optimization step. Choosing a fixed step size to advance $\phi$ quickly leads to poor results. However, for dense meshes, the computational cost of line search algorithms severely outweigh their potential benefits. Consequently, we will determine the step size $\alpha$ using 
\\
\begin{equation}
    \alpha = \alpha_0\frac{||\phi||_{\infty}}{||V_n||_\infty},
\end{equation}
\\
where $\alpha_0 \in \mathbb{R^+}$ is a predetermined multiplier per batch of iterations. The ratio of $L^\infty$ norms ensures a relatively consistent advancement of $\phi$, and thereby stable rootfinding behavior in the procedure to follow. It is especially helpful when $V_n$ is poorly resolved due to potential cusps or changes in topology. Larger values of $\alpha_0$ are selected initially when the approximate shape of a minimizer is unclear. Then, during the refinement process, $\alpha_0$ is gradually reduced alongside $h_{max}$, the largest permitted edge size in the triangulation mesh. 
\\\\
We are now prepared to discuss the domain update procedure. Define a cross section as the set of discretized points along a line normal to $\partial \Omega$. The number of points in a cross sections is equal to the number of layers in our narrow band including the boundary. After updating $\phi$ with a step size $\alpha$, we can compute $\frac{\partial\phi}{\partial n}$ at each point using the appropriate finite differencing scheme. As informational tools, we will also compute the forward differences of $\text{sgn}(\phi)$ and $\text{sgn}\left(\frac{\partial \phi}{\partial n}\right)$ across the cross sections for all but the outermost points.
\\\\
Using these results, we first look for discretized points such that the forward difference of $\text{sgn}(\phi)$ is non-zero. This gives us candidate intervals for interpolating a solution to $\phi(\mathbf{x})=0$. 
\\\\
Alternatively, for cross sections where there is no change in $\text{sgn}(\phi)$, we check if there is a change in $\text{sgn}(\frac{\partial \phi}{\partial n})$ across the cross section. If there is not, we select the appropriate direction for extrapolating a solution to $\phi(\mathbf{x})=0$. 
\\\\
For cross sections that fail both criteria, linear approximations are clearly insufficient, and thus we do not attempt to make them. This leaves us with two cases:
\\
\begin{enumerate}
    \item $\mathbf{Interpolation.}$ First, let $\mathbf{x}$ be the discretized point where the forward difference of $\text{sgn}(\phi)$ changed values. Then, using the central difference result for $\frac{\partial \phi}{\partial n}(\mathbf{x})$,  compute a new boundary point as 
    $$\mathbf{x}^*=\mathbf{x}-\dfrac{\phi(\mathbf{x})}{\frac{\partial \phi}{\partial n}(\mathbf{x})}\ \hat{\mathbf{n}},$$
    \\
    where $\hat{\mathbf{n}}$ is the exterior unit normal vector from the closest boundary point.
    \\
    \item $\mathbf{Extrapolation.}$ The methodology for our extrapolation is very similar to our interpolative step, however there is a key difference in order to limit error. Unlike the interpolative step, in which $\phi$ is known to cross $0$ within the interval of interest under assumptions of continuity, far stronger and most likely inaccurate assumptions would have to be made on $\phi$ to state that it must cross $0$ within the extrapolation interval. Additionally, the error in linear approximation will balloon for large changes in $\mathbf{x}$, as the updated value of $\phi$ is nonlinear. Thus, we choose the following as our extrapolation rule, noting that $\mathbf{x}$ in this case represents either the innermost or outermost point in a cross section:
    $$\mathbf{x}^*=\mathbf{x}-\left[\text{sgn}\left(\frac{\partial \phi}{\partial n}(\mathbf{x})\right)\cdot\min\left(h_{min},\left|\dfrac{\phi(\mathbf{x})}{\frac{\partial \phi}{\partial n}(\mathbf{x})}\right|\right)\right] \hat{\mathbf{n}}.$$
\end{enumerate}

With rootfinding complete, we now move on to domain construction. Select all layer points with $|\phi(\mathbf{x})| < \epsilon$, where $\epsilon$ is chosen as a small constant; for our implementation, we selected $\epsilon=10^{-3}$. Next, concatenate these points with the list of new boundary points computed with the algorithm above, as well as the interior nodes not selected as part of the narrow band. This is our new representative set of points. Before we can move on to the next iteration, however, we will first apply some beneficial postprocessing.
\\\\
These manipulations require the use of alpha shapes. A survey of the details and applications of these objects can be found in \cite{edelsbrunner2011}.
\\\\
Our first objective is to ease changes in topology that result in the creation of disjoint regions. Using an alpha shape, we split our points based on the disjoint regions they represent, if any. The alpha shape is created using $3(h_{max})^2$ as the maximum area for filling internal holes and $3(h_{min})^2$ as the maximum area of regions to be suppressed. These account for erroneous approximations during rootfinding. Then, for each region, we orient the points around their principle axes using a singular value decomposition and perform the following:
\begin{enumerate}[i]
    \item Determine which axis has the most mass concentrated along it. We will call this the primary axis. 
    \item Moving along the primary axis, split the region into several disjoint subregions of roughly equal length.
    \item For each subregion, compute the maximum difference along the secondary axis between any two discretized points in that batch. We will call this the batch width.
    \item Discard any subregions with a batch width smaller than the average edge in the FEM mesh, approximated via $\frac{h_{min}+h_{max}}{2}$.
\end{enumerate}

This process hastens the removal of vestigial features from connected regions that are in the process of splitting into disjoint regions, and generally results in quicker refinement of boundaries undergoing sharp change.
\\\\
Finally, we reconstruct an alpha shape using our new set of points, and rescale the region to ensure that $|\Omega|=1$. This helps improve numerical stability when working with disjoint unions of sets or conducting detailed refinement of candidate minimizers. 
\\\\
This marks the end of the update procedure, and thus the next iteration can begin. In summary, the evolution process is the following:
\begin{enumerate}
    \item Given a representative set of points for $\Omega$, create a mesh using a Delaunay triangulation.
    \item Compute the FEM solution to the eigenvalue problem on the mesh.
    \item Compute the FEM solution to $\Delta \phi=1$ on the mesh with Dirichlet boundary conditions. This solution is our LSF, $\phi$.
    \item Create a narrow band of points to consider in the level set by moving points on $\partial \Omega$ normal to the boundary and in proportion to $h_{min}$.
    \item Compute the deformation field $V_n$ on the narrow band based on the objective function $J_p$ and the solutions from the above steps. $V_n$ is properly defined on $\partial \Omega$, extended to $\Omega^-$ through direct evaluation with interpolated quantities, and extended to $\Omega^+$ via a simplified velocity extension method.
    \item Advect $\phi$ by $V_n$ as per the Hamilton-Jacobi equation with a step size proportional to the ratio of $||\phi||_\infty$ and $||V_n||_\infty$, as shown in (16).
    \item Perform rootfinding via a combination of interpolation and limited extrapolation.
    \item Post-process and rescale the new representative set of points to allow for smoother changes in topology.
\end{enumerate}
\vspace{12 pt}
To begin the iteration process, we either select a disk, a previous connected minimizer, or some perturbation of either. This perturbation is conducted by deforming the boundary using a $\sin$ function of low magnitude and random angular frequency $\omega\in[1,3]$. In the refinement process, we gradually lower $h_{max}$ and $\alpha_0$, increase $p$, and occasionally apply the above perturbation to escape local minima.
\bigskip
\section{Theoretical Results for the Regularized Objective Function}
\bigskip
The numerical techniques discussed above were reliant on the assumption that minimizing our regularized approximation to $\lambda_k$ using the $p$-norm, i.e. $(5)$, would nonetheless produce true minimizers of $\lambda_k$ as $p\to \infty$. We also asserted that such an approximation was well suited to the minimization of non-simple eigenvalues due to an invariance across orthonormal bases for eigenfunctions. This section is dedicated to formally proving those notions.
\\
\begin{theorem}
    The derivative of $J_p$ with respect to a deformation of the domain, i.e. (8), is invariant under orthonormal changes of basis for eigenfunctions. 
\end{theorem}
\begin{proof}
 Consider the symmetric bilinear form
 $$B(u,v) = \int_{\partial \Omega} \frac{\partial u}{\partial n} \frac{\partial v}{\partial n} V\cdot n \; d\sigma,$$
 where $\Omega \subset \mathbb{R}^N$, $u,v\in L^2(\Omega)$, $V$ is a deformation field of $\mathbb{R}^N$, and $n$ is the unit normal vector. 
 \\\\
 In (8), eigenvalues are enumerated with multiplicity, so equal eigenvalues $\lambda_i=\lambda_j$ implies equal weights $\mu_i=\mu_j$. Thus, (8) contains terms of the form
 $$\mu_k \sum_{i=1}^mB(u_i,u_i),$$
 where $\mu_k$ is the common weight and $m$ is multiplicity of the eigenvalue.
 \\\\
 An orthonormal change of basis for the eigenvectors can be written as
 $$v_i=\sum_{j=1}^mQ_{ij}u_j,$$
 where $\mathbf{Q}$ is the change of base matrix.
 \\\\
 Thus,
 $$\sum_{i=1}^m B(v_i,v_i)=B \left(\sum_{j=1}^mQ_{ij}u_j,\sum_{s=1}^mQ_{is}u_s\right)=\sum_{i}\sum_{j,s}Q_{ij}Q_{is}B(u_j,u_s).$$
 \\
 Now, since $\mathbf{Q}$ is an orthogonal matrix,
 $$\sum_iQ_{ij}Q_{is}=\delta_{js,}$$
 where $\delta_{js}$ is the Kronecker delta.
 \\\\
 Applying that information gives us
 $$\sum_{i=1}^m B(v_i,v_i)=\sum_{j,s}\delta_{js}B(u_j,u_s)=\sum_{i=1}^m B(u_i,u_i).$$
 \\
 Therefore, sums of $B(u_j,u_j)$ across a basis are invariant under orthonormal changes of basis. Since (8) only depends on its choice of eigenfunctions insofar as it depends on this sum, it too is invariant under orthonormal changes of basis.
\end{proof}
\bigskip
\bigskip
We now move on to the task of proving that minimizers of $J_p$ converge to the minimizers of $J$ as $p\to\infty$. In order to do so, we must first introduce a few notions of convergence. In the following, $Y$ denotes a separable metric space. 
\\
\begin{definition}
Given a sequence ($G_n$) of functionals from $Y$ into $\overline{\mathbb{R}}$ we say that ($G_n$) $\Gamma$-converges to a functional $G$ if for every $y\in Y$ we have:
\begin{enumerate}[i]
    \item $\forall y_n \to y \;\;\;\;G(y)\leq \liminf_{n\to+\infty}G_n(y_n)$
    \item $\exists y_n \to y \;\;\;\;G(y)\geq \limsup_{n\to+\infty}G_n(y_n)$
\end{enumerate}
See \cite{dalmaso1993} for further details on $\Gamma$-convergence.
\end{definition}
\begin{definition}
    Consider a set $U$ with no topological structure given a priori, and a state functional $G:Y\to\mathbb{\overline R}$ with the properties:
    \begin{enumerate}[i]
        \item for every $u \in U$ the function $G(u,\cdot)$ is lower semicontinuous in the space $Y$
        \item $G$ is equi-coercive in the sense that there exists a coercive lower semicontinuous functional $\Psi: Y\to \mathbb{\overline R}$ such that
        $$G(u,y) \geq \Psi(y) \quad \forall u \in U, \; \forall y \in Y$$
        \item the mapping $\Gamma_G(u) = G(u,\cdot)$ is one-to-one
    \end{enumerate}
    \medskip
    Then, we say that $u_n \to u$ in $U$ if functionals $G(u_n,\cdot)$ $\Gamma$-converge to $G(u,\cdot)$ in $Y$. This convergence on $U$ will be called $\gamma$-convergence.  
\end{definition}
Moving closer to the technicalities of our specific discussion, we will need to introduce the notion of the Sobolev capacity of a subset $E$ in $\mathbb{R}^N$ defined by 
$$\text{cap}(E)=\inf\{\int_{\mathbb{R^N}}|\nabla u|^2 + u^2 dx : u\in \mathcal{U}_E\},$$
where $\mathcal{U}_E$ is the set of all functions $u$ of the Sobolev space $H^1(\mathbb{R}^N)$ such that $u\geq1$ almost everywhere in a neighborhood of $E$. 
\\\\
Then, a subset $A$ of $\mathbb{R}^N$ is said to be quasi-open if for every $\varepsilon > 0$ there exists an open subset $A_\varepsilon$ of $\mathbb{R^N}$ such that cap$(A_\varepsilon \triangle A)\ < \varepsilon$, where $\triangle$ denotes the symmetric difference of sets.
\\\\
In a similar vein, a function $f: D\to \mathbb{R}$ is said to be quasi-continuous if for every $\varepsilon>0)$ there exists a continuous function $f_\varepsilon:D\to\mathbb{R}$ such that cap$(\{f \neq f_\varepsilon\}) < \varepsilon$.
\\\\
For the rest of our definitions and results, we will now consider the admissible class 
\\
$$\mathcal{A}=\{A\subseteq D : A \text{ is quasi-open}\},$$
\\
where $D$ is a bounded open subset of $\mathbb{R}^N$.
\\\\
Although the original problem from (1) is not constrained to a bounded $D$, it is actually quite reasonable for us to focus on the bounded case. It has been proven that, even in the general setting with $\mathbb{R}^N$ as the domain, all minimizers of (2) are bounded and have finite perimeter \cite[Theorem 3.3]{Bucur2012}. Furthermore, numerical approaches will always use an implicit $D$ by necessity. Thus, constraining the problem to $D$ is natural for our purposes and helps simplify later arguments. 
\\\\
Moving on, denote by $w_A$ the solution to the torsion problem
$$
\begin{cases}
-\Delta w_{A} = 1 & \text{in } A,
\\
w_A = 0 \in H^1_0(A).
\end{cases}
$$
\\
\begin{definition}
We say that a sequence ($A_n$) of $\mathcal{A}$ weakly $\gamma$-converges to $A \in \mathcal{A}$ if $w_{A _n}$ converges weakly in $H^1_0(D)$ to a function $w \in H^1_0(D)$ that we may take quasi-continuous such that $A = \{w > 0\}$.
\end{definition}
Note that $w$ does not necessarily coincide with $w_A$ unless $A_n$ $\gamma$-converges to $A$, and that taking $w$ to be quasi-continuous implies that $A$ is always quasi-open \cite[Definition 5.3.1]{Bucur2005-kd}. Since we treat $-w_{A_n}$ as LSFs in our numerical work, these stipulations are not a concern; only elements of the zero set of $w$ in proximity to $\{w > 0\}$ are relevant to numerically reconstructing $A$. That is to say, within a numerical context, $-w$ is a practically optimal level set for $A$. 
\\
\begin{theorem}
As $p \to \infty$, the volume-constrained minimizers of (5) over $\mathcal{A}$ weakly $\gamma$-converge to the volume-constrained minimizers of (2). 
\end{theorem}
\begin{proof}
    Consider the functionals 
$$F_p(\Omega)\equiv \left(\sum_{j=1}^{k}\lambda_j(\Omega)^p\right)^{\frac{1}{p}}$$
and
$$F(\Omega)\equiv\lambda_k(\Omega),$$
\\
where $\Omega \in \mathcal{A}$ and $\lambda_k$ is the $k$th Dirichlet eigenvalue. These correspond to computational objective functions (5) and (2) respectively.
\\\\
It is known that the mapping $\Omega \mapsto\lambda _k(\Omega)$ is monotonically decreasing with respect to set inclusion and also continuous with respect to the $\gamma$-convergence \cite[Definition 3.3.1, Example 5.4.2]{Bucur2005-kd}. Since the mapping for each $\lambda_k$ has these properties individually, their $p$-norm, $F_p$, must have them as well. Consequently, each $F_p$ must also admit a volume-constrained minimizer over $\mathcal{A}$  \cite[\text{Theorem 5.4.1}]{Bucur2005-kd}. Let us call the minimizer to this problem $W_p$.
\\\\
Next, due to the fact that the weak $\gamma$-convergence on $\mathcal{A}$ is sequentially compact \cite[Proposition 5.3.4]{Bucur2005-kd}, there exists a set $V\in\mathcal{A}$ such that, up to a subsequence, (which we will not relabel for clarity), $W_{p}\xrightarrow[]{w\gamma}V$. We will return to this convergence later.
\\\\
Now that we know volume-constrained minimizers of $F_p$ exist, we can also show that the minimum value of $F_p$ converges to the minimum value of $F$ as $p\to \infty$. Let $W$ be a volume-constrained minimizer for $F$, i.e, $\forall\Omega \in \mathcal{A}$
$$\min_{|\Omega|=1}F=F(W).$$
Then, it is a matter of simple point-wise convergence that
$$\lim_{p\to\infty}F_p(W)=F(W),$$ 
which leads to 
$$\limsup_{p\to\infty}\min_{|\Omega|=1} F_p \leq \lim_{p\to\infty}F_p(W) = F(W).$$
\\
However, by the properties of the $p$-norm, we also know that 
$$F_p(\Omega)\geq F(\Omega)$$
for all $p$ and $\Omega$, which necessitates
\begin{equation}
\lim_{p\to\infty}\min_{|\Omega|=1} F_p = \min_{|\Omega|=1}F.
\end{equation}
\\\\
To complete our preparations, we will now show that $F$, and by consequence $F_p$, is weak $\gamma$-lower semicontinuous. Consider the weak $\gamma$-convergence of quasi-open sets $A_n\xrightarrow[]{w\gamma}A$. Then, there exists a sequence of open sets $G_n\subseteq D$ such that, up to a subsequence, $G_n\xrightarrow[]{\gamma}A$ and $A_{n}\subseteq G_n$ \cite[Lemma 4.3.15, Definition 5.3.1]{Bucur2005-kd}. The $\gamma$-continuity of $F$ gives us
$$F(A)\leq\liminf_{n\to\infty}F(G_n).$$
\\
Then, since $F$ is monotonically decreasing with respect to set inclusion, we also have
$$F(A)\leq\liminf_{n\to\infty}F(G_n)\leq\liminf_{n\to\infty}F(A_n),$$
\\
which by definition implies $F$ is weak $\gamma$-lower semicontinuous. 
\\\\
This lets us move on to our final result. Returning to the fact that $W_{p}\xrightarrow[]{w\gamma}V$, by the weak $\gamma$-lower semicontinuity of $F$, we have
$$F(V)\leq \liminf_{p\to\infty}F(W_p).$$
Then, since $F_p(\Omega) \geq F(\Omega)$ for all $p$ and $\Omega$, we can also state that
$$F(V)\leq \liminf_{p\to\infty} F_p(W_p).$$
\\
Now, since the elements $W_p$ are defined as the minimizers for $F_p$, and we know from (\theequation) that the volume-constrained minima of our functionals are equal under a limit, it follows that
$$F(V)\leq  \liminf_{p\to \infty} \min_{|\Omega|=1} F_p =\min_{|\Omega|=1}F.$$
\\
Fortunately, we know that $|V| \leq 1$ from the weak $\gamma$-lower semicontinuity of the Lebesgue measure \cite[Proposition  5.3.6]{Bucur2005-kd}. Therefore, since $F$ is monotonically decreasing with respect to dilations of set volume, the optimality of $F(V)$ implies that $V$ must minimize $F$ subject to the volume constraint.
\end{proof}
\bigskip
\section{Disjoint Regions} {
\bigskip
In this section, we discuss the problem of finding optimal disjoint minimizers for our objective functions of interest. These can be interesting to study even if they do not produce an apparent minimum. Furthermore, optimal disjoint minimizers are relatively simple to find; they are simply combinations of existing connected minimizers.    
\\\\
Recall we defined $\lambda^*_k$ as the minimal $k$th Dirichlet eigenvalue across all $\Omega \subset \mathbb{R}^N$ such that $|\Omega|=1$. Now, assume that $\Omega^*_k$ is the union of at least two disjoint sets with positive measure. Then:
\begin{equation}
(\lambda_k^*)^{N/2} = (\lambda_i^*)^{N/2} + (\lambda_{k-i}^*)^{N/2}
= \min_{1 \le j \le (n-1)/2} \left( (\lambda_j^*)^{N/2} + (\lambda_{k-j}^*)^{N/2} \right),
\end{equation}
where $i$ is a value of $j$ that satisfies the minimum.
\\ 
Furthermore, 
\begin{equation}
    \Omega^*_k=\big[ ( \frac{\lambda^*_i}{\lambda^*_k})^{1/2} \; \Omega^*_i \big] 
    \cup 
    \big[ ( \frac{\lambda^*_{k-i}}{\lambda^*_k})^{1/2} \; \Omega^*_{k-i} \big].
\end{equation}
\\
The proofs for these statements can be found in \cite{henrot2006}.
\\\\
Together, (17) and (18) allow us to compute the optimal disjoint minimizer for an arbitrary $\lambda_k$ given minimizers for $\lambda_1,\lambda_2,...,\lambda_{k-1}$. The accuracy of this resulting minimizer of course depends on the accuracy of the preceding minimizers used to compute it. Nevertheless, this process provides us with a powerful tool. 
\\\\
We now extend this result to find the optimal disjoint minimizers of $J$ and $J_p$ respectively in the two dimensional case ($N=2$). First, consider the simple case of $J(\Omega)=|\Omega|\lambda_k(\Omega)$. This proof justifies some of our choices in computational implementation. Given the initial condition $|\Omega|=1$, we can rewrite (17) as 
\\
$$\lambda_k^*|\Omega^*_k| = \lambda_i^*|\Omega^*_i| + \lambda_{k-i}^*|\Omega^*_{k-i}|
= \min_{1 \le j \le (n-1)/2}  \lambda_j^*|\Omega^*_j| + \lambda_{k-j}^*|\Omega^*_{k-j}|, $$
which, then, by the definition of $J$, becomes
\\
\begin{equation}
    J(\Omega^*_k) = J(\Omega^*_i) + J(\Omega^*_{k-i})
= \min_{1 \le j \le (n-1)/2} J(\Omega^*_j) + J(\Omega^*_{k-j}).
\end{equation}
\\
Using similar reasoning, we will rewrite (18) as 
\\
$$
\Omega^*_k=\left[ \left( \frac{\lambda^*_i|\Omega^*_i|}{\lambda^*_k|\Omega^*_k|}\right)^{1/2} \; \Omega^*_i \right] 
    \cup 
    \left[ \left( \frac{\lambda^*_{k-i}|\Omega^*_{k-i}|}{\lambda^*_k|\Omega^*_k|}\right)^{1/2} \; \Omega^*_{k-i} \right],
$$
which then becomes
\\
\begin{equation}
    \Omega^*_k=\left[ \left( \frac{J(\Omega^*_i)}{J(\Omega^*_k)}\right)^{1/2} \; \Omega^*_i \right] 
    \cup 
    \left[ \left( \frac{J(\Omega^*_{k-i})}{J(\Omega^*_{k})}\right)^{1/2} \; \Omega^*_{k-i} \right].
\end{equation}
\\
Finding a minimizing pair for (19) proves preferable to the volume-constrained (17), as it allows us to account for numerical error in the recomputation of $|\Omega|$. This benefit also applies when computing (20) as opposed to (18).
\\\\
Now, we consider minimizing $J_p(\Omega) = |\Omega|(\sum_{n=1}^{k} (\lambda_n)^p)^{\frac{1}{p}}$.
First, note that, if $\Omega$ is a disjoint union of the form $\Omega=\Omega_i\cup\Omega_j$, then
$$J_p(\Omega)=(|\Omega_i|+|\Omega_j|)(\sum_{n=1}^i\lambda_{n,i}^p+\sum_{n=1}^j\lambda_{n,j}^p)^{1/p},$$
where we use $\lambda_{n,i}$ to denote the $nth$ Dirichlet eigenvalue on $\Omega_i^*$, the region  associated with $\lambda_i^*$.
\\\\
Our task then becomes
$$\min\;\; (c_i|\Omega_i|+c_j|\Omega_j|)(\frac{1}{c_i^p}\sum_i\lambda_i^p+\frac{1}{c_j^p}\sum_j\lambda_j^p)^{1/p}.$$
Taking the gradient with respect to the coefficients and setting it equal to $\bf{0}$ gives us
$$\frac{1}{p}(\frac{1}{c_i^p}\sum_i\lambda_i^p+\frac{1}{c_j^p}\sum_j\lambda_j^p)^{1/p-1}\cdot(\frac{-p}{c_i^{p+1}}\sum_i\lambda_i^p)=-|\Omega_i|(\frac{1}{c_i^p}\sum_i\lambda_i^p+\frac{1}{c_j^p}\sum_j\lambda_j^p),$$
and
$$\frac{1}{p}(\frac{1}{c_i^p}\sum_i\lambda_i^p+\frac{1}{c_j^p}\sum_j\lambda_j^p)^{1/p-1}\cdot(\frac{-p}{c_j^{p+1}}\sum_i\lambda_i^p)=-|\Omega_j|(\frac{1}{c_i^p}\sum_i\lambda_i^p+\frac{1}{c_j^p}\sum_j\lambda_j^p).$$
Since we know the relevant terms to be nonzero, we may divide the two and obtain
$$\frac{c_j^{p+1}\sum_i\lambda_i^p}{c_i^{p+1}\sum_j\lambda_j^p}=\frac{|\Omega_i|}{|\Omega_j|}.$$
By imposing our volume constraint $c_i|\Omega_i|+c_j|\Omega_j|=1$, we can rearrange and receive
$$\left(\frac{1-c_i|\Omega_i|}{c_i|\Omega_j|}\right)^{p+1}=\frac{|\Omega_i|\sum_j\lambda_j^p}{|\Omega_j|\sum_i\lambda_i^p},$$
and
$$\left(\frac{1-|\Omega_i|c_i}{c_i|\Omega_i|}\right)^{p+1}=\frac{|\Omega_j|^p\sum_j\lambda_j^p}{|\Omega_i|^p\sum_i\lambda_i^p}.$$
Now, let
\begin{equation}
    \chi = \left(\frac{|\Omega_j|^p\sum_j\lambda_j^p}{|\Omega_i|^p\sum_i\lambda_i^p}\right)^\frac{1}{p+1}=\left(\frac{J_p(\Omega_j)}{J_p(\Omega_i)}\right)^\frac{p}{p+1}.
\end{equation}
\\
This quantity can easily be calculated from a pair of saved minimizers. 
\\
Then, our coefficients simplify to the following:
$$c_i=\frac{1}{|\Omega_i|\left(1+\chi\right)},$$
and
$$c_j=\frac{\chi}{|\Omega_j|\left(1+\chi\right)}.$$
\\
Accordingly, we can write the optimal disjoint union for minimizing $J_p$ as
\begin{equation}
    \Omega_{k,p}^*=\left[\left(\frac{1}{|\Omega_i|\left(1+\chi\right)}\right)^{\frac12}\Omega_{i,p}^*\right]\cup\left[\left(\frac{\chi}{|\Omega_j|\left(1+\chi\right)}\right)^{\frac12}\Omega_{j,p}^*\right].
\end{equation}
\\\\
Both of these results were converted into algorithms and used to compare optimal disjoint regions to the best known connected minimizers. In general, all minimizers of $J_p$ for lower values of $p$ are connected, likely due to the influence of the first eigenfunction. The frequency of disjoint minimizers for $J_p$ increases as $p$ increases. Results for the standard objective function $J$, meanwhile, are visible in the next section.
}

\section{Numerical Results}
\bigskip
\begin{center}
\rmfamily{TABLE 1. Numerical Results for the Minimal Dirichlet Eigenvalues with $|\Omega|=1$}
\\
\vspace{.25 cm}
\setlength{\arrayrulewidth}{0.5mm}
\setlength{\tabcolsep}{18pt}
\renewcommand{\arraystretch}{1}
\begin{tabular}{ |p{.2cm}||p{1.25cm}|p{1.65cm}|p{1.65cm}|  }
 \hline
 $k$ & Classical Level Set FEM \cite{Oudet2010} & MFS (Best Known Results) \cite{AntunesFreitas2012} & Laplace LSF FEM \\
 \hline
 5 & 78.47 & 78.20 & 78.20 \\
 6 & 88.96 & 88.52 & 88.50 \\
 7 & 107.47 & 106.14 & 106.14 \\
 8 & 119.9 & 118.90 & 118.92 \\
 9 & 133.52 & 132.68 & 132.45 \\
 10 & 143.45 & 142.72 & 142.72 \\
 11 & - & 159.39 & 159.43 \\
 12 & - & 172.85 & 172.92 \\
 13 & - & 186.97 & 186.96 \\ 
 14 & - & 198.96 & 199.02 \\
 15 & - & 209.63 & 209.64 \\
 16 & - & - & 227.94 \\
 17 & - & - & 241.71 \\
 18 & - & - & 255.70 \\
 19 & - & - & 268.87 \\ 
 20 & - & - & 285.44\\
 \hline
\end{tabular}
\\
\end{center}
\begin{figure}
    \centering
    \caption{Minimizers 1-25 in Row Major Order}
    \makebox[\textwidth]{\includegraphics[width=.6\paperwidth]{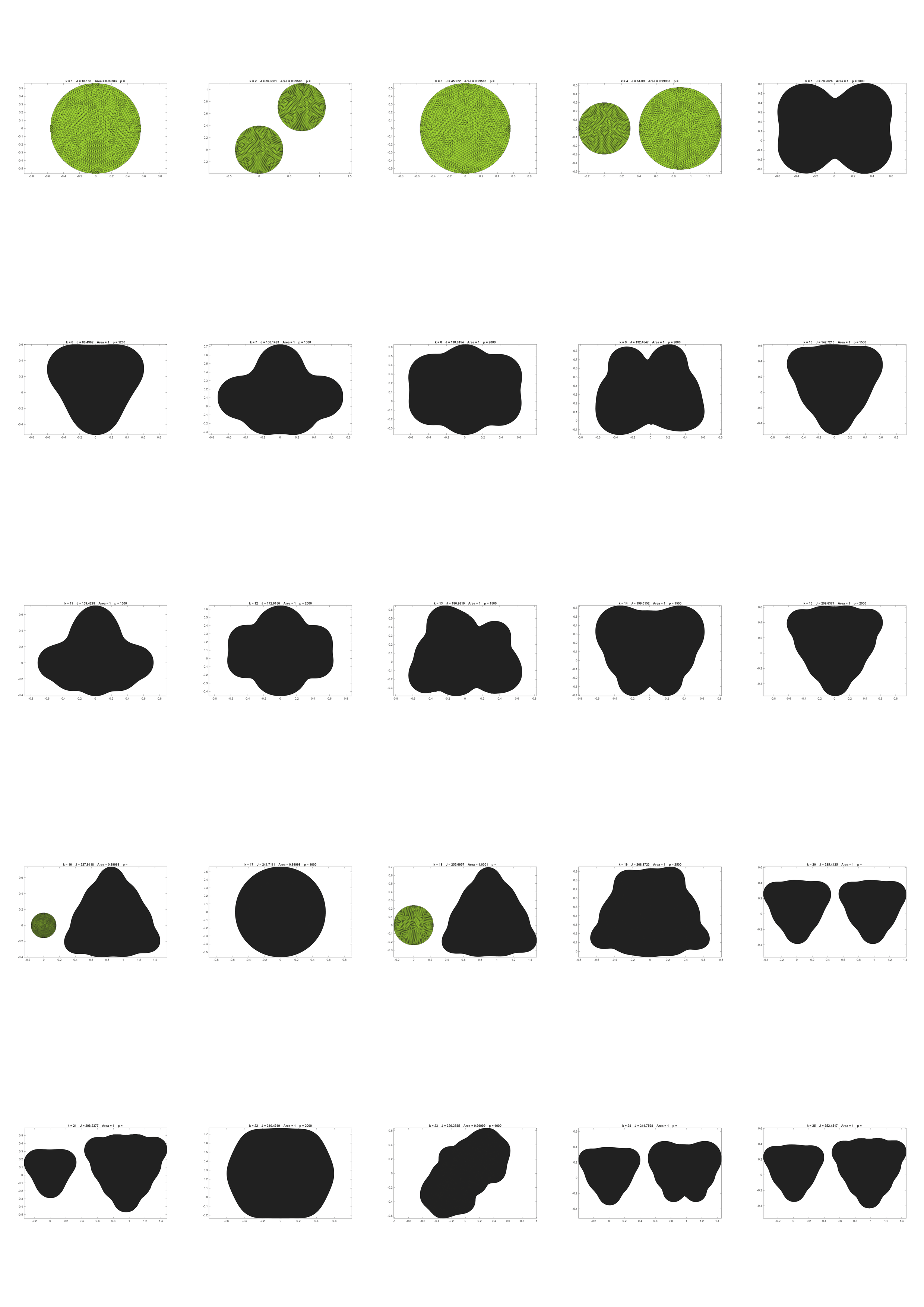}}
    \label{fig:min_1-25}
\end{figure}
\begin{figure}
    \centering
    \caption{Minimizers 26-50 in Row Major Order}
    \makebox[\textwidth]{\includegraphics[width=.6\paperwidth]{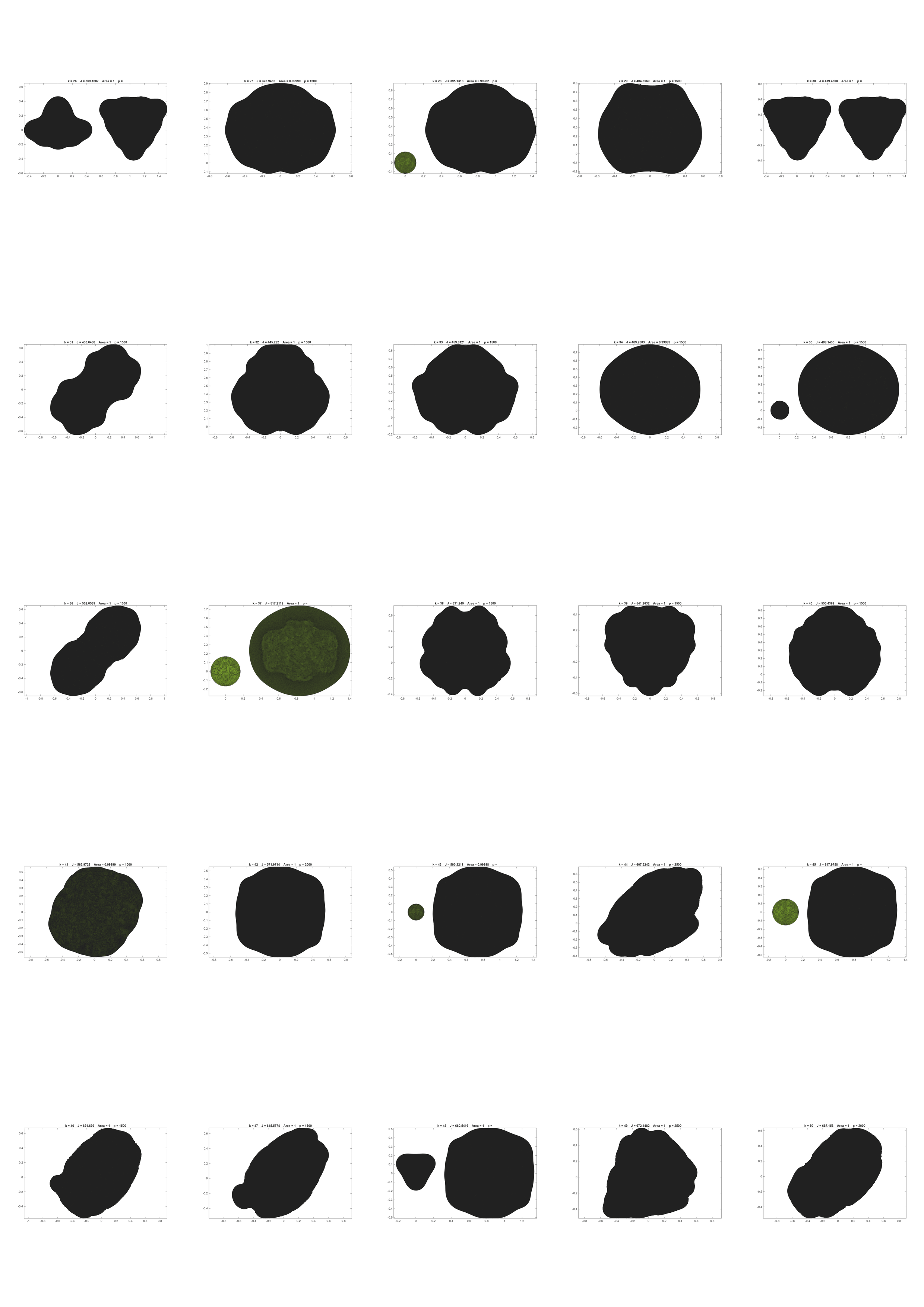}}
    \label{fig:min_26-50}
\end{figure}
\bigskip
\bigskip
With the aid of our general results, we may now investigate some more specific problems within the field.
\bigskip
\subsection{\space \space Cusps on the Boundary.}

As briefly mentioned earlier, the boundary of a minimizer, $\partial \Omega^*$, admits a decomposition into three regions: a regular and smooth component $\partial^*\Omega^*$, a set of singular points $Z_{AC}$, and a set of cusp points $Z_C$ \cite{kriventsov2017regularitydegenerate}. The authors who proved this decomposition focused primarily on the reduced boundary and $Z_{AC}$, as little could be said about the last set, $Z_C$, except that if it is not empty, then, depending on the rate of opening of the cusp, it may also not be regular \cite{kriventsov2017regularitydegenerate}. Since analytical methods for investigating this property are limited, analyzing the numerical results presented in this work may provide additional insight. Based on the Euler-Lagrange equation, such cusp points may be identified by locating boundary points where all of the eigenfunctions associated with the $k$th eigenvalue and its multiplicities vanish. The zero sets for $k=13$ make it the most likely candidate for a minimizer that exhibits a cusp on the boundary\footnote{All zero sets can be viewed at \url{}{https://github.com/AtharvThakur0/Laplace-LSF-Dirichlet-Eigenvalue-Numerics/tree/main/ZeroSetPlots}}.
\\

\begin{figure}
    \centering
    \caption{Uppermost Eigenfunction Zero Sets for the Minimizer of $\lambda_{13}$}
    \makebox[\textwidth]{\includegraphics[width=.8\paperwidth]{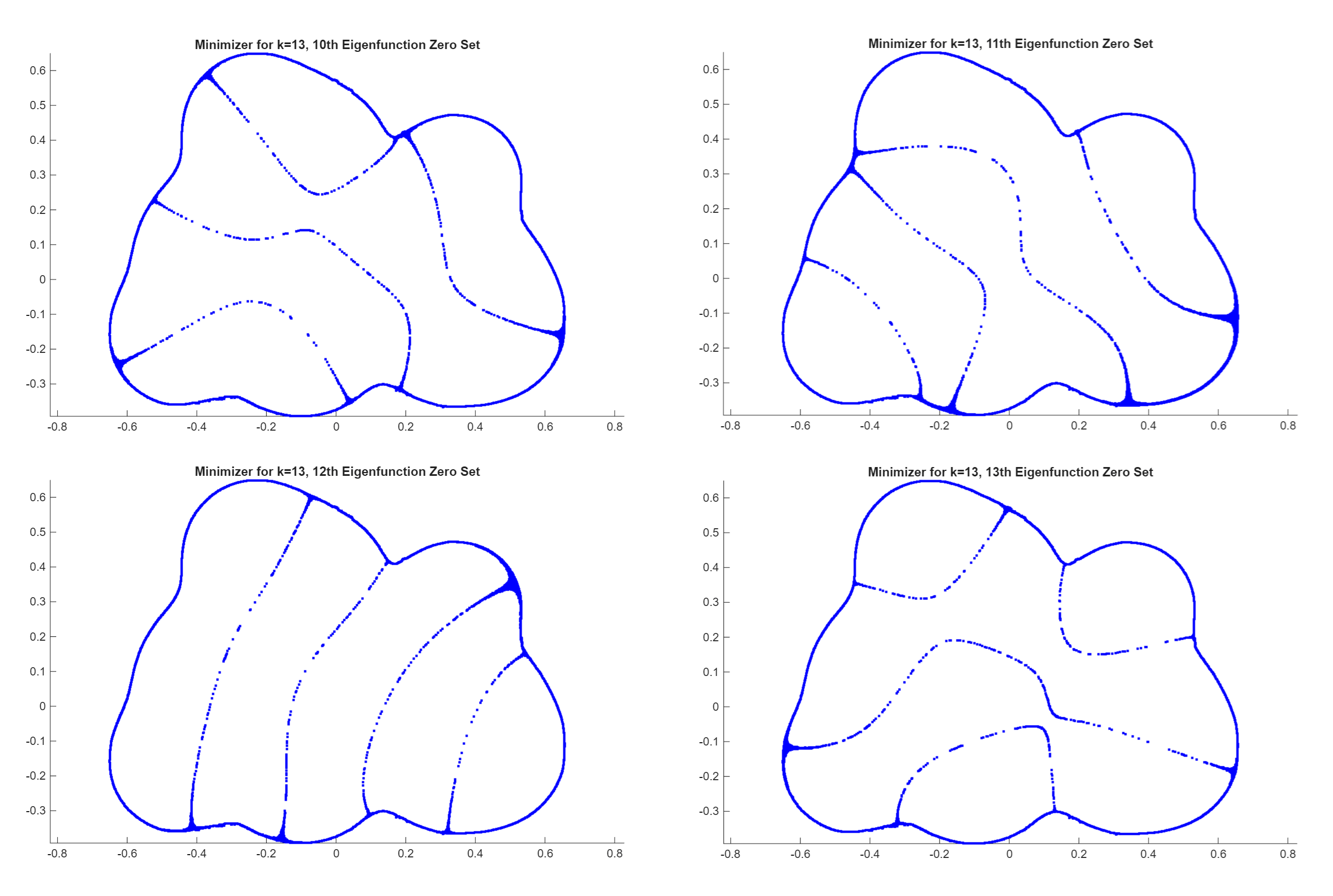}}
    \label{fig:zs13}
\end{figure}

The 5 uppermost computed eigenvalues on this region were 160.55, 186.57, 186.81, 186.89, and 186.96. The slight drift away from the candidate cusp at the top across the upper four zero sets could be attributed to the numerical error between values, as the theoretical minimizer for $k=13$, and indeed any $k>1$, likely has a non-simple $k$th eigenvalue.
\\\\
Another possible interpretation is that $Z_C$ is in fact empty, at least for $1\leq k\leq50$. Although minimizing $J_p$ over $J$ results in smoother shapes, which may be hiding potential cusps, attempting to minimize $J_p$ for larger $p$ or directly minimize $J$ did not yield superior results. Thus, we can be fairly certain that our computational results are more or less representative of the true shapes --- or at least local minima --- and limited primarily by computational precision. 
\\\\
As a last remark, we note that our results largely avoid a common class of local minima: disks or disjoint unions of disks. Given $\Omega \subset \mathbb{R}^2$, such regions cannot minimize $\lambda_k$ for $k\geq5$ \cite[Theorem 6]{Berger2015}. Besides of $k=17$, our results successfully reflect that guarantee. For $k< 5$, meanwhile, our results supported the possibility that $\lambda_3$ is minimized by the disk and $\lambda_4$ is minimized by a union of disks with known radii as per (18) \cite[Theorem 6]{Berger2015}.
\medskip
\subsection{\space \space Testing the Pólya Conjecture.}
Consider the Dirichlet eigenvalue counting function
$$\mathcal{N}_\Omega^{Dir}(\lambda):= \#\{\lambda_k(\Omega): \lambda_k(\Omega) \leq \lambda\}.$$
Weyl's law \cite{Weyl1911} gives asymptotics for this function that depend only on the set volume:
$$\mathcal{N}_\Omega^{Dir}(\lambda)=\frac{|\Omega|\cdot\omega(n)}{(2\pi)^n}\lambda^{n/2} + o(\lambda^{n/2}).$$
In the above, $\omega(n)$ denotes the volume of the unit ball in $\mathbb{R}^n$.
\\\\
George Pólya later conjectured that this term is, in fact, a uniform bound on bounded open domains \cite{Polya1954PlausibleInference}. For our regions of interest with $\Omega \subset \mathbb{R}^2$, this would mean
$$\mathcal{N}_\Omega^{Dir}(\lambda) < \frac{|\Omega|}{4\pi}\lambda.$$

Since the minimizers for the Dirichlet eigenvalues are the natural extreme cases for this conjecture, we will test it by calculating the quantities
\begin{equation}
\mathcal{N}_{\Omega^*_k}^{Dir}(\lambda_k^*) - \frac{|\Omega^*_k|}{4\pi}\lambda_k^*.    
\end{equation}
\begin{figure}
    \centering
    \caption{Difference Between the Dirichlet Counting Function and Weyl Term}
    \smallskip
    \includegraphics[width=.75\paperwidth]{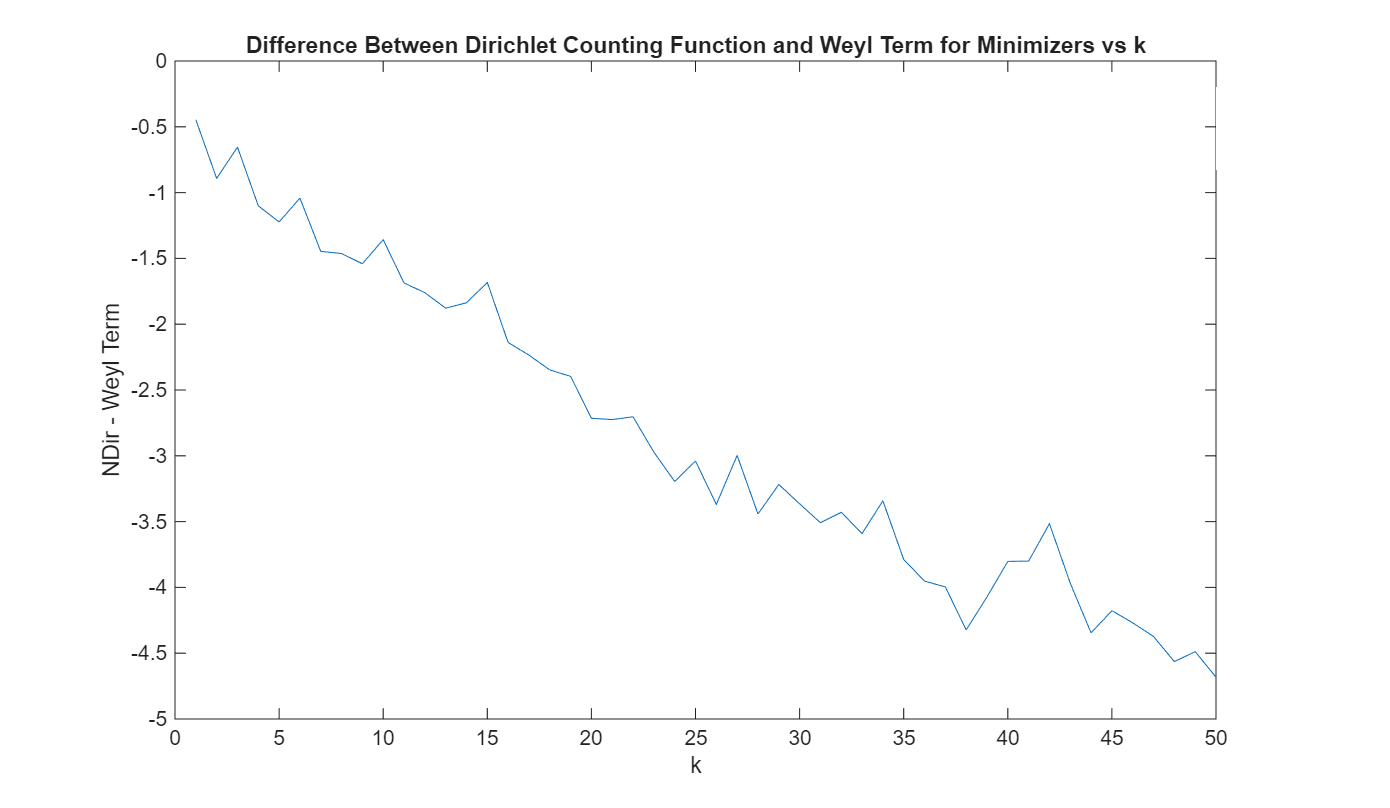}
    \label{fig:polyatestgraph}
\end{figure}
\\
As expected, the Pólya conjecture holds for $1\leq k \leq 50$. Furthermore, it appears that the difference between the counting function and the Weyl Term is positively correlated with $k$. This may be of interest to those who study the problem in greater depth or work on proving related bounds. A graph of this experiment\footnote{The exact values for the computed differences can be obtained from the codebase at \url{https://github.com/AtharvThakur0/Laplace-LSF-Dirichlet-Eigenvalue-Numerics}} has been included in Figure 4.
\\
\subsection{\space \space Convergence of Weights.} 

A key component of the stationary conditions for the minimization of $J_p$ are the weights assigned to each eigenfunction $u_1,u_2,...u_k$ on $\Omega$ by the derivative of $J_p$. Referring back to (15) and applying the volume constraint $|\Omega|=1$, the weight $\mu_m$ given to each $\left(\frac{\partial u_m}{\partial n}\right)^2$ can be written as
\begin{equation}
    \mu_m=\left(\frac{\lambda_m}{\left(\sum_{j=1}^{k}\lambda_j^p\right)^\frac{1}{p}}\right)^{p-1}
\end{equation} 
\\
Holding eigenvalues $\lambda_1, \lambda_2,...,\lambda_k$ constant yields
\\
\begin{equation}
\lim_{p\to\infty}\mu_m=\frac{[\lambda_m=\lambda_k]}{\nu},
\end{equation}
\\
where $[\cdot]$ is the Iverson bracket and $\nu$ is the multiplicity of $\lambda_k$. 
\\\\
However, the far more interesting and difficult question is to ask how the weights $\mu_m$ dynamically behave across $J_p$-minimizers, $\Omega^*_{k,p}$, as $p\to\infty$. Figures 5 and 6 feature a selection of experiments on this question. We only select $k$ such that the baseline computational minimizers of $J$, $\Omega^*_k,$ are connected regions. This removes instances where changes in topology would severely distort the apparent behavior of weights due to our limited sampling of $p$.

\begin{figure}
  \caption{Weights vs p for $k\in[5,6,7,8]$}
  \bigskip
  \includegraphics[width=.75\paperwidth]{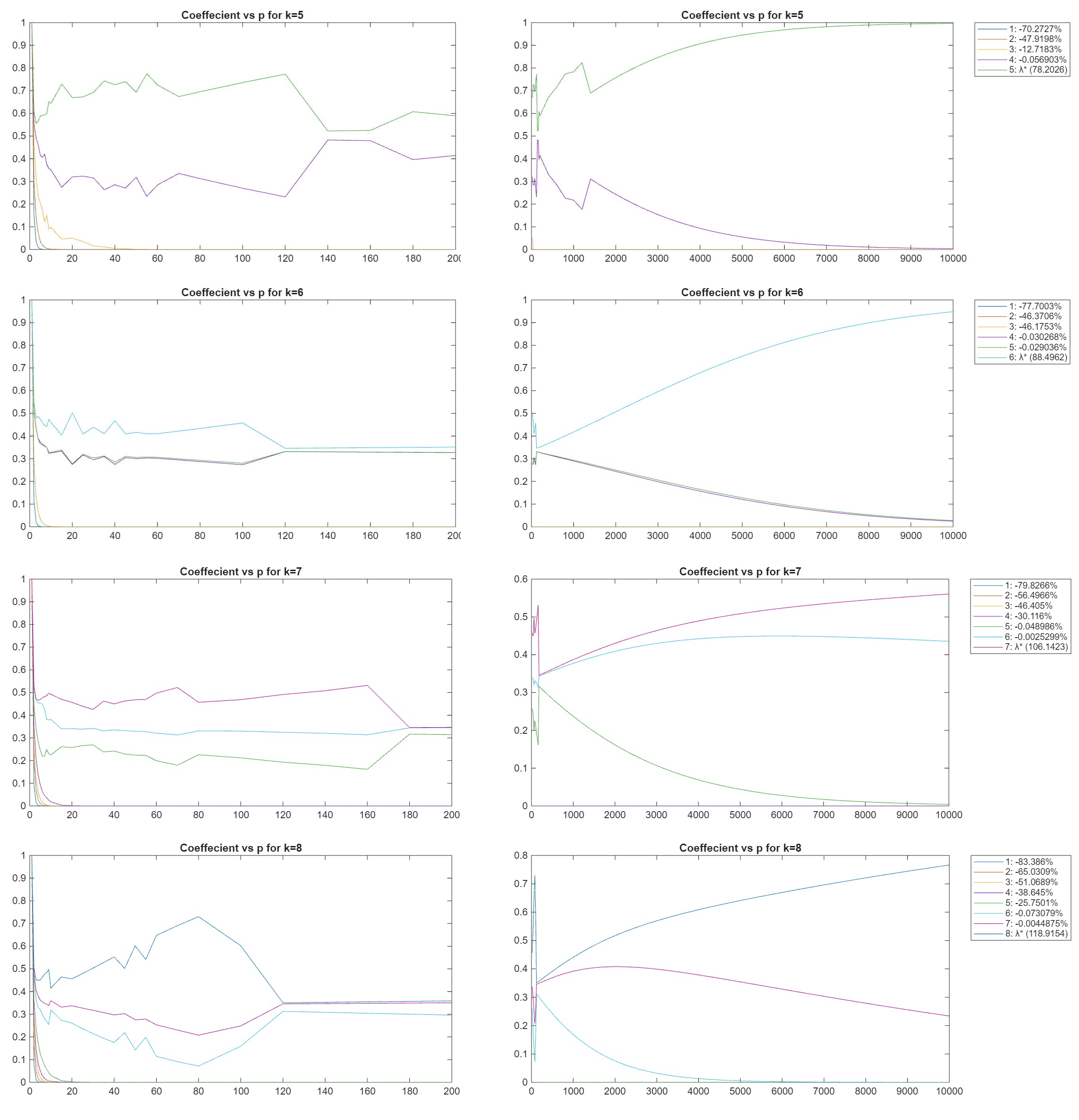}
  \label{fig:weights_v_p_k_5_8}
\end{figure}
\begin{figure}
  \caption{Weights vs p for $k\in[9,10,15,17,22]$}
  \bigskip
  \includegraphics[width=.66\paperwidth]{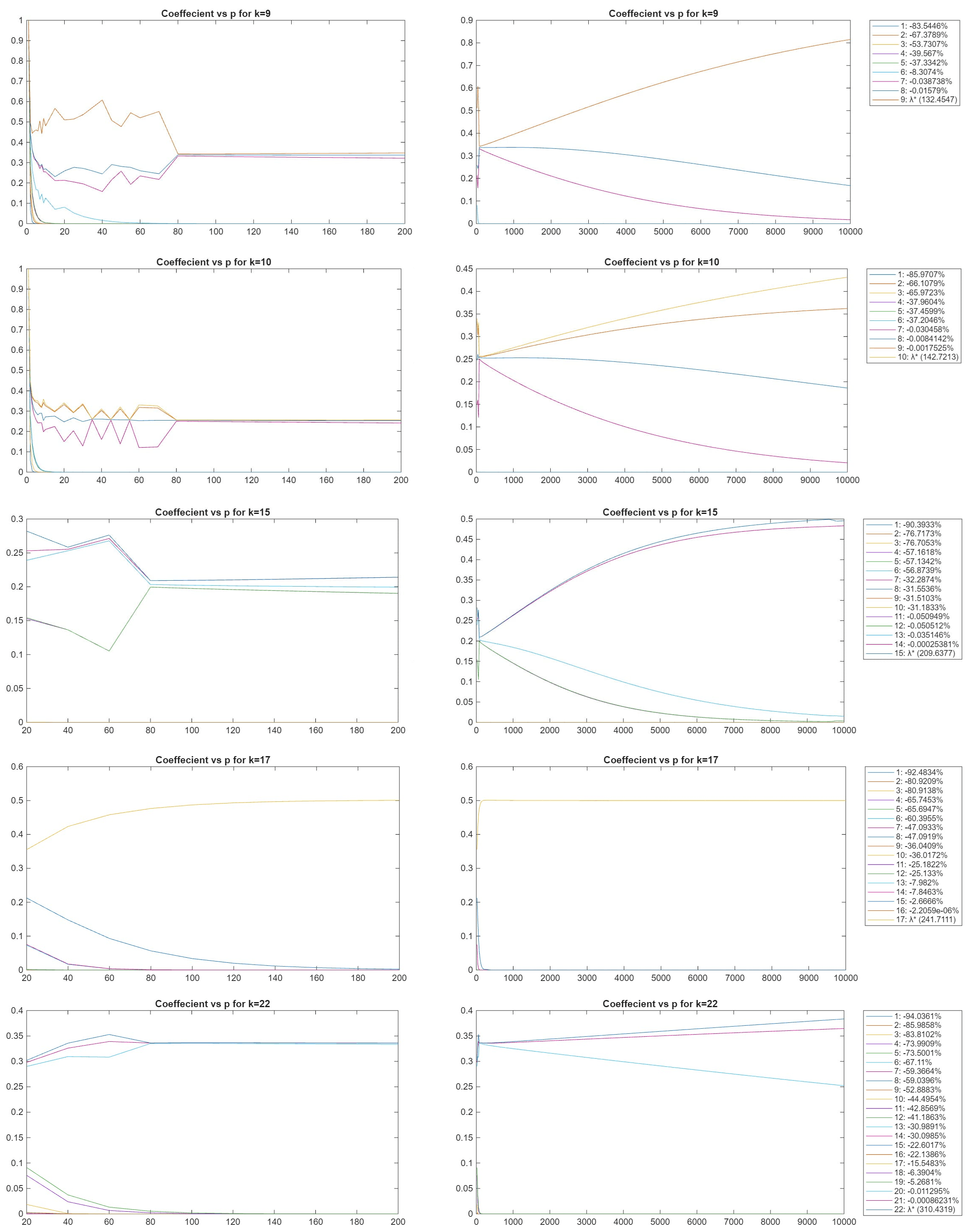}
  \label{fig:weights_v_p_k_9_22}
\end{figure}
\bigskip
Clearly, numerical error distorts convergence behavior in many of these graphs. However, for those that display low percent differences between their uppermost eigenvalues, it seems as though $\mu_m$ tend towards the values predicted by (\theequation) before diverging based on the discrepancy between $\lambda_m$ and $\lambda_k$. Nevertheless, we cannot say for certain, as early behavior is quite noisy and the $p$-norm begins to dominate any changes in the eigenvalues themselves beyond $p\approx 3000$. This problem warrants further investigation, either through analytical means or the use of increased computing power; we refined each $J_p$ minimizer up to $h_{max}=.005$ due to resource constraints, however it would be feasible and potentially fruitful go beyond this precision. 
\bigskip
\bigskip

\section{Conclusion}{
\bigskip
In this work, we made several improvements to numerical methods for computing volume-constrained minimizers for the Dirichlet eigenvalues. Our primary innovations were the formulation of a new type of level set function based on the torsion function and a new method of domain representation designed to accompany it. These techniques were developed alongside a regularized objective function and an optimal disjoint union finding algorithm. We proved that the minimizers of our regularized problem weak $\gamma$-converge to the minimizers of the original --- a notion of convergence strong enough to produce behaviorally equivalent level sets for the task of numerical domain reconstruction. With these implementations, we produced FEM-based results competitive with the overall best known results from the literature. Our techniques were then applied to investigating open problems within the field that may see benefit from numerical experiments. All of the code and results from this work have been made publicly available\footnote{\url{https://github.com/AtharvThakur0/Laplace-LSF-Dirichlet-Eigenvalue-Numerics}} so as to facilitate further analysis and encourage the use of the Laplace LSF in other problems where it may provide benefit.}
\bigskip
\section{Acknowledgments}
{
\bigskip
The author would like to thank Professor Dennis Kriventsov for his extensive guidance. Accordingly, this work was partially supported by NSF DMS grant 2247096.
}
\\
\bibliographystyle{siam} 
\bibliography{refs}
\end{document}